\documentclass[11pt]{amsart}

\usepackage{amsmath,amssymb,amsthm}
\usepackage{geometry}
\newtheorem{lemma}{Lemma}
\newtheorem{claim}{Claim}

\newtheorem{proposition}{Proposition}
\newtheorem*{mainthm}{Main Theorem}
\newtheorem*{thme}{Theorem 2}

\begin{document}
	\author{Ning Yu}
	\address{Department of Mathematics, Northwestern University, 2033 Sheridan Road, Evanston, IL 60208, USA}
	\title{The Porous Medium Equation and Convexity}
	\maketitle
	
	\begin{abstract}
		We show that root convexity of the pressure is preserved by the porous medium equation if the initial data is root convex, positive, smooth, bounded below by a quadratic function and has bounded Laplacian. This complements a result of Daskalopoulos-Hamilton-Lee on root concavity of pressure.
	\end{abstract}
	\section{Introduction}
	The porous medium equation (PME) is a model for gas diﬀusing in a
	porous medium. The gas density $u(x,t) \ge 0$ satisfies the equation
	\[
	u_t = \Delta (u^m)  \quad \text{on} \quad \mathbb{R}^d \times [0,\infty),
	\tag{1.1}
	\]
	where $m>1$. The pressure $v$ is defined by $v = mu^{m-1}$. The equation satisfied by $v$ is 
	\[
	v_t = v\Delta v + (1/(m-1))|\nabla v|^2 \quad \text{on} \quad \mathbb{R}^d \times [0,\infty).
	\tag{1.2}
	\]
	An important weak solution to $(1.1)$ is the Barenblatt solution \cite{barenblatt1952unsteady}:
	\[
	U(x,t):= t^{-\alpha}\left(\left(C-\frac{\alpha(m-1)}{2md}t^{-2\alpha/d}|x|^2\right)_{+}\right)^{1/(m-1)}.
	\]
	where $\alpha = d/(d(m-1)+2)$ and $C>0$. 
	Then $V = mU^{m-1}$ is quadratic in $x$ in the set $\{V>0\}$ for every $t>0$. Obviously the concavity of $V$ is preserved on its support for all time.
	On the other hand, the solution starting at the quadratic function
	\[
	v_0 = a|x-b|^2+c \quad (a>0, c\ge 0, \,\, \text{and $b \in \mathbb{R}^d$ is arbitrary})
	\]
	is given by
	\[
	v(x,t)=	\frac{a|x-b|^2}{1 - a(2d+4/(m-1))t} + c \left(1 - a(2d+4/(m-1))t\right)^{-2d/(2d+4/(m-1))},
	\]
	which will also be quadratic, therefore convex, as long as the solution exists. However the preservation of concavity or convexity is not always true. When $d = 1$, the preservation of concavity was proved in \cite{BenilanVazquez1987} by Bénilan-Vázquez.
	In \cite{DaskalopoulosHamiltonLee}, 
	Daskalopoulos-Hamilton-Lee proved that if the initial data $v_0 = v(x,0)$ is root concave (in the sense that $D^2(\sqrt{v_0}) \le 0$) and smooth up to boundary, and satisfies the nondegeneracy condition $v_0+|Dv_0|^2\ge c>0$, then $v(\cdot,t)$ is root concave up to boundary for all $0\le t < \infty$. This implies that the free boundary is smooth for all time (\cite{Daskalopoulos1998Regularity}, \cite{DaskalopoulosHamiltonLee}). Ishige-Salani proved in \cite{ishige2010convexity} that the convexity of free boundary breaks in short time.   Chau-Weinkove proved in \cite{chau2024nonpreservation} that when $d = 2$, $\alpha$-concavity ($v^{\alpha}$ is concave) is not preserved by the pressure for $0\le\alpha<1/2$ and $1/2<\alpha\le 1$ (0-concavity means that $\log v$ is concave) and Shen-Talla \cite{shenTalla2026} proved this for $d \ge 3$.
	
	Concavity or convexity properties of solutions to other parabolic problems have been considered also in [6,8,10,11,12,16,18].

	It is natural to ask whether a classical solution exists and remains root convex if the initial data is smooth and root convex. The existence and uniqueness of solution of $(1.1)$ starting with growing initial data with certain growth conditions were established in \cite{benilan_crandall_porous} by Bénilan-Crandall-Pierre. They proved that an optimal condition on the initial data that guarantees the existence of a $L^1_{\text{loc}}$ weak solution of $(1.1)$ is 
	\[
	\|u_0\|_r
	:= \sup_{R>r} \frac{1}{R^{\gamma}} \int_{\{|x|\le R\}} u_0(x)\,dx < \infty,
	\tag{1.3}
	\]
	where $\gamma = d + 2/(m-1)$. The necessity of this condition was proved in \cite{AronsonCaffarelli1983} by Aronson-Caffarelli.
	If $u_0$ satisfies $(1.3)$, a solution $u(x,t)$ of (1.1) exists for
	\[
	t \in [0,T), \qquad T(u_0) = \frac{c_1}{\lim_{r \rightarrow \infty}\|u_0\|^{m-1}_r},
	\]
	where $c_1$ is a constant depending only on $m$ and $d$. Observe that $\|u_0\|_r$ decreases in $r$ so the limit exists. We will state this result in detail in Section 2. We prove the following theorem:
	\begin{mainthm}
		Suppose that the density $u_0 := u(x,0)$ satisfies $(1.3)$, and the pressure $v_0$ is positive, smooth, root convex ($D^2(\sqrt{v_0}) \ge 0$), and satisfies
		\[
		K:=\sup_{x\in \mathbb{R}^d}\Delta v_0 < \infty \quad \text{and} \quad v_0>c(|x|^2+1)
		\]
		for some $c>0$. Then there exists a smooth classical solution $v(x,t)$ that is root convex for $t \in [0, T)$ where $T = \min \{(m-1)/((m+1)K), T(u_0)\}$.
	\end{mainthm}
	It is easy to find functions satisfying the assumptions on the initial data. For example, $(f+\sqrt{|x|^2+ \varepsilon})^2$ where $f$ is compactly supported in $B_R$ and
	\[
	D^2f>-\frac{\varepsilon}{(R^2+\varepsilon)^{3/2}}
	\quad\text{on }B_R.
	\]
	The solution of (1.2) starting with a quadratic function $a|x-b|^2+c \,\, (a>0, c>0)$ is clearly root convex throughout its interval of existence.
	\\
	
	The source of difficulty is that the maximum principle method in [15] requires the solution to be compactly supported, but in the case of growing initial data, the solution may not be compactly supported. To apply the maximum principle method, observe that 
	\[
	D^2_{ij}(\sqrt{v}) = \frac{1}{2\sqrt{v}}\left(v_{ij}-\frac{v_iv_j}{2v}\right).
	\]
	If $v \rightarrow \infty$ as $|x| \rightarrow \infty$ and $v_{ij}-v_iv_j/2v$ is bounded then $D^2_{ij} (\sqrt{v})V^iV^j \rightarrow 0$ for every $|V| = 1$ as $|x| \rightarrow \infty$. Therefore the problem is mostly reduced to finding bounds for $|D^2v|$ and $|Dv|^2/v$. We now give an outline of the proof of Main Theorem. In Section 3.1, we derive estimates for $v_0$ and $|Dv_0|^2$ and construct a sequence of smooth cutoff functions. These cutoff functions are then used to construct suitable approximations of the initial data. The solutions corresponding to the approximating initial data are constant for sufficiently large $|x|$. In Section 3.2, we apply maximum principle arguments to prove that $|D^2v|$ and $|Dv|^2/v$ are bounded for short time for the solutions starting at the approximating functions and then take limits. This allows us to prove that $D^2(\sqrt{v}) \ge 0$ for short time using the maximum principle method. In Section 3.3 we show that this short amount of time can be extended to a time depending only on $m$, $d$, $\sup_{\mathbb{R}^d} \Delta v_0$ and the maximal time for which the solution exists.
	\section*{AI Disclosure}
	The author used ChatGPT to check calculation and typos. 
	\section{Existence of solution}
	In this section we state the existence and uniqueness result in \cite{benilan_crandall_porous}. Define
	\[
	X = \{f \in L^1_{\text{loc}}(\mathbb{R}^d): f \ge 0 \,\,  \text{and} \,\,  \|f\|_r < \infty \,\, \text{for some }r \ge1\}.
	\]
	\begin{thme}
		Let $u_0 \in X$. Then, there exists a constant $c_1$ depending only on $m$ and $d$ and a unique solution that satisfies
		
		(i) $u\in C([0,T(u_0)); L^1_{\text{\text{loc}}}(\mathbb{R}^d)) \cap L^{\infty}_{\text{\text{loc}}}([0,T(u_0));X)$,
		
		(ii) $u_t = \Delta(u^m)$ in $\mathcal{D}'(\mathbb{R}^d \times [0,T(u_0)))$ (meaning that 
		\[
		\int\int_{\mathbb{R}^d\times (0,\infty)} u \phi_t \, dxdt= -\int\int_{\mathbb{R}^d\times (0,\infty)} u^m \Delta \phi \, dxdt
		\] for every test function $\phi \in C^{\infty}_c(\mathbb{R}^d \times (0,\infty))$),
		
		(iii) $u(0) = u_0$,
		
		where 
		\[
		T(u_0) = \frac{c_1}{\lim_{r \rightarrow \infty}\|u_0\|^{m-1}_r}.
		\]
	\end{thme}
	The existence result follows from the following estimate:
	\begin{proposition}
		Let $u \in L^{\infty}(\mathbb{R}^d)$ be nonnegative, $C \in (0,\infty)$ and 
		\[
		\Delta u^{m-1} \ge -C \quad  \text{in} \quad  \mathcal{D}'(\mathbb{R}^d)
		\]
		(meaning that 
		\[
		\int_{\mathbb{R}^d} u^{m-1} \Delta \phi \, dx \ge -\int_{\mathbb{R}^d} C\phi \, dx
		\] for every test function $\phi>0$, $\phi \in C^{\infty}_c(\mathbb{R}^d)$). Then there is a constant $K$ depending only on $d,m$ such that for $1 \le r \le R$ 
		\[
		\frac{1}{R^2}\|u\|^{m-1}_{L^\infty(B_R)} \le K(C^{\lambda (m-1)}\|u\|_r^{2\lambda(m-1)/d} + \|u\|_r^{m-1}).
		\tag{2.1}
		\]
		where $\lambda = d/(d(m-1)+2)$.
	\end{proposition}
	Now we explain briefly how this estimate implies the existence of solution. Suppose that $u_0$ satisfies assumptions in Main Theorem.
	Define
	\[
	u_{0,n}(x) = 
	\begin{cases}
		\min\{u_0(x),n\} \quad \text{if} \quad |x| \le n\\
		0 \quad \text{if} \quad |x| > n
	\end{cases}
	\]
	By Chapter 9 of \cite{Vazquez} there exists a unique solution $u_n \ge 0$ (in the sense that $u_t = \Delta(u^m)$ in $\mathcal{D}'(\mathbb{R}^d \times (0,\infty))$ and $u(x,t) \rightarrow u_0$ as $t\rightarrow 0$ in $L^1$) to the problem
	\[
	\begin{cases}
		u_t = \Delta(u^m)\quad  \text{in}\quad \mathbb{R}^d\times (0,\infty),\\
		u(x,0) = u_{0,n}(x)\quad \text{in} \quad \mathbb{R}^d.
	\end{cases}
	\]
	The Aronson-Bénilan inequality \cite{AronsonBenilan1979} $\Delta u_{n}^{m-1} \ge -\sigma/t$ where $\sigma = m\alpha/(m-1)$ and Proposition 1 imply that 
	\[
	\frac{1}{R^2}\|u_n\|^{m-1}_{L^\infty(B_R)} \le K(t^{-\lambda (m-1)}\|u_n\|_r^{2\lambda(m-1)/d} + \|u_n\|_r^{m-1}).
	\]
	for every $t>0$. (Since $u_n(\cdot,t) \in L^{\infty}(\mathbb{R}^d)$ for every $t$, we can apply Proposition 1 to $u_n(\cdot, t)$.) By Lemma 12.4 in $[20]$ we have 
	\[
	R^{-2/(m-1)}\|u_n(t)\|_{L^\infty (B_R(0))} \le c t^{-\lambda }\|u_{0,n}\|_r^{2\lambda/d}
	\tag{2.2}
	\]
	\[
	\text{for} \quad 1\le r\le R, \quad 0<t<T_{r}(u_0) := c_1/\|u_0\|_r^{m-1}, 
	\]
	for some constants $c$, $c_1$ depending only on $m$ and $d$. 
	Since $u_{0,n}$ is increasing in $n$, by comparison principle, $u_n$ is increasing in $n$. By estimate $(1.3)$ and $(2.2)$, $u_n(t)$ is locally uniformly bounded for $0<t<T_r(u_0)$, so $\lim_{n\rightarrow \infty}u_n(x,t) := u(x,t)$ exists for $0<t<T_r(u_0)$. Therefore $u$ is a solution in the sense that 
	
	\[
	\int\int_{\mathbb{R}^d\times (0,T_r(u_0))} u \phi_t \, dxdt+ \int\int_{\mathbb{R}^d\times (0,T_r(u_0))} u^m \Delta \phi \, dxdt = -\int_{\mathbb{R}^d} u_0 \phi(x,0) \, dx
	\] for every test function $\phi \in C^{\infty}_c(\mathbb{R}^d \times (0,T_r(u_0)))$. Then, since $u>0$, by standard parabolic theory $u$ is a classical solution of $(1.1)$ on $\mathbb{R}^d \times [0,T_r(u_0))$ for every $r\ge1$. Finally let $r \rightarrow \infty$. 
	
	\section{Proof of Main Theorem}
	
	\subsection{Construction of smooth cutoff functions}
	In this section we construct a sequence of smooth cutoff functions $\phi_n$ which we will need later to approximate $v_0$. First we prove 
	the following estimates for $v_0$ and $|Dv_0|$. 
	\begin{lemma}
		Suppose that $v_0$ satisfies the assumptions of Main Theorem. Then there exists a constant $c$ depending only on $v_0, m$ and $d$ such that $v_0 \le c (1+|x|^2)$ and $|Dv_0| \le c(1+|x|)$.
		\begin{proof}
			By root convexity and boundedness of $\Delta v_0$ we have $0\le D^2v_0\le  K$  and $|Dv_0|^2/v_0 \le 2\Delta v_0 \le 2K$. Then by Taylor's formula
			\[
			v_0(x) \le v_0(0) + Dv_0(0)\cdot x \, + \frac{K}{2}|x|^2 \le c(1+|x|^2).
			\]
			Moreover $|Dv_0| \le \sqrt{2Kv_0} \le c(1+|x|)$. The constant $c$ may change from line to line.
		\end{proof}
	\end{lemma}

	\begin{proposition}
		There exists a sequence of smooth cutoff functions $\phi_n \in C^\infty_c(\mathbb{R}^d)$ that satisfy conditions
		
		(i) $\phi_n = 1$ for $|x| \le n$.
		
		(ii) $0<\phi_n < 1$ for $n<|x| < n+R(n)$ for some $R(n)$ depending on $n$.
		
		(iii) $\phi_n = 0$ for $|x|\ge n+R(n)$.
		
		(iv) $|D\phi_n|^2 \le \phi_n/(|x|^2+1)$.
		
		(v) $|D^2\phi_n| \le C/(|x|^2+1)$ for some $C$ independent of $n$.
		\begin{proof}
			Define
			\[
			\eta(t) = 
			\begin{cases}
				0 \quad \text{if} \quad  t \le 0 \\
				e^{-1/t} \quad \text{if} \quad t > 0,
			\end{cases}
			\]
			and
			\[
			H(s) = \frac{\eta(s)}{\eta(s)+\eta(1-s)}.
			\]
			Then $H = 0$ when $s \le 0$ and $H = 1$ when $s >1$, and $H$ is smooth. 
			Define
			\[
			\psi_n(x) = 
			\begin{cases}
				1 \quad \text{if} \quad |x| \le  \log((n^2+1)^{1/2}) \\
				H(( \log((n^2+1)^{1/2})+L-|x|)/L)^2  \quad \text{if} \quad  \log((n^2+1)^{1/2})<|x|\le  \log((n^2+1)^{1/2})+L\\
				0 \quad \text{if} \quad |x|> \log((n^2+1)^{1/2})+L
			\end{cases}                 
			\]
			where $L = 4 \sup_{s \in \mathbb{R}}|H'(s)|+1$. In $( \log((n^2+1)^{1/2}), \log((n^2+1)^{1/2})+L) \cup (- \log((n^2+1)^{1/2})-L,- \log((n^2+1)^{1/2}))$, 
			\[
			(\psi_n')^2(x)/\psi_n(x) = 4H'\left(\frac{\log((n^2+1)^{1/2})+L-x}{L}\right)^2/L^2< 1.
			\]
			Define $\phi_n=\psi_n(\log((|x|^2+1)^{1/2}))$. Then $\phi_n$ satisfy these conditions:
			
			\[
			|\nabla \phi_n|^2
			=
			\frac{|x|^2}
			{(1+|x|^2)^2}
			\left[
			\psi_n'\, \left(\log\left((1+|x|^2)^{1/2}\right)\right)
			\right]^2
			\]
			
			\[
			|\nabla \phi_n|^2/\phi_n
			=
			\frac{|x|^2}
			{(1+|x|^2)^2}
			\frac{\left[
				\psi_n'\!\left(\log\left((1+|x|^2)^{1/2}\right)\right)
				\right]^2}{\psi_n(\log((|x|^2+1)^{1/2}))}
			\]

			\[
			\begin{aligned}
				(\phi_n)_{ij}
				&=
				\frac{\psi_n''\!\left((1/2){\log(1+|x|^2)}\right)x_i x_j}
				{(1+|x|^2)^2}
				\\
				&\quad
				+\psi_n'\!\left((1/2)\log(1+|x|^2)\right)
				\left[
				\frac{\delta_{ij}}
				{(1+|x|^2)} - \frac{2x_ix_j}{(1+|x|^2)^2}\right].\\
			\end{aligned}
			\]
			Therefore
			\[
			|D\phi_n|^2/\phi_n \le 1/(|x|^2+1),
			\]
			and
			\[
			|D^2\phi_n| \le C/(|x|^2+1).
			\]
			for some $C$ independent of $n$ because $\sup_{x\in \mathbb{R}} |\psi_n'|$ and $\sup_{x\in \mathbb{R}} |\psi_n''|$ are both independent of $n$.
		\end{proof}
	\end{proposition}
	
	\subsection{Preservation of root convexity}
	Assume that 
	\[
	\sup_{x\in \mathbb{R}^d}|{v_0}_{ij}|^2 < \infty \quad \text{and}  \quad v_0\ge c(|x|^2+1),
	\]
	for some $c>0$.
	In this section we are going to show that there exists a $\tau>0$ depending on the initial data such that the matrix inequality
	\[
	D^2_{ij}(\sqrt{v}) \ge 0
	\]
	is preserved for $t \in [0,\tau]$.  \\
	
	First observe that if the initial data $v_0$ is bounded below by a quadratic function $c|x|^2$ then there exists a $c'>0$ such that the solution is bounded below by $c'|x|^2$ for all time.
	\begin{lemma}
		If $v_0 \ge c|x|^2$ for some $c>0$, then there exists a $c'>0$ such that $v(x,t) \ge c'|x|^2$  for every $(x,t) \in \mathbb{R}^d \times [0,\tau]$.
		\begin{proof}
			We compare the solution of the porous medium equation
			\[
			w(x,t) = \frac{|x|^2}{C-bt},
			\] 
			with $v(x,t)$,
			where 
			\[
			b = 2d + 4/(m-1),
			\]
			and $C$ is a positive number. Choose $C > b\tau$ and $C>1/c$. Then $w(x,0) \le v(x,0)$. By comparison principle, $v(x,t)\ge w(x,t) \ge |x|^2/C$.
		\end{proof}
		
	\end{lemma}
	
	Now we prove that  $|Dv|^2/v$ and $|D^2v|$ are bounded in $\mathbb{R}^d \times [0,\tau]$.
	Define $v_{0, n}^{(\varepsilon)} := v_0\phi_n+\varepsilon$ where $\phi_n $ are the smooth cutoff functions constructed in section 3.1.
	In the following two lemmas we denote the solution $v_{n}^{(\varepsilon)}$ starting at $v_{0, n}^{(\varepsilon)}$ temporarily by $v$.
	\vspace{\baselineskip}
	\begin{lemma}
		There exists $\tau'>0$ and $A>0$ such that
		\[
		(v_i)^2 - Av < 0
		\]
		in $\mathbb{R}^d \times [0,\tau']$.
		\begin{proof}
			Consider
			\[
			Q = (v_i)^2-A(t) v
			\]
			where $A(t)$ is a function of $t$ to be determined.
			At $t = 0$, 
			\begin{align*}
				( v_{0n})_i^2 
				&=(v_0\phi_n)_i^2 = ({v_0}_i)^2\phi_n^2 + 2({v_0})_i{(\phi_n)}_iv_0\phi_n + v_0^2({\phi_n}_i)^2\\
				&\le({v_0}_i)^2\phi_n + C_1v_0\phi_n + v_0^2\phi_n/(|x|^2+1)\\
				&\le C_2{v_0}\phi_n + C_1v_0\phi_n + C_3v_0\phi_n\\
				&< A(0)(v_0\phi_n+\varepsilon) = A(0)(v_{0n}+\varepsilon)
			\end{align*}
			for some $A(0)$ independent of $\varepsilon$ and $n$. The first inequality follows from the calculation of $\nabla \phi_n$ in section 3.1 and that $\phi_n$ can be chosen to satisfy $|\nabla \phi_n|^2 < \phi_n/(|x|^2+1)$. The second inequality follows from that ${v_0}_i^2<Cv_0$. As $|x| \rightarrow \infty$, $v_i^2 \rightarrow 0$ so $Q<0$. 
			Suppose $Q = 0$ for the first time at $(x_0,t_0)$. At this point, we have
			\begin{align*}
				Q_t &= vQ_{kk} + 2rv_kQ_k + 2v_i^2v_{kk} - 2vv^2_{ik} - A'v + rA^2v \\
				&\le 2A vv_{kk} - 2vv^2_{ik} - A'v + rA^2v\\
				&= 2v(A v_{kk}-v_{ik}^2)-A'v+rA^2v\\
				&\le 2v(Av_{kk}-(v_{kk})^2/d)-A'v+rA^2v\\
				&\le 2dA^2v+rA^2v - A'v  	 
			\end{align*}
			The last inequality is because if $Av_{kk}-(v_{kk}^2/d)>0$ then $v_{kk}<dA$.
			Then we can choose
			\[
			A' > (r+2d)A^2.
			\]
			or 
			\[
			A(t) = \frac{1}{1/{A(0)}-(r+2d+1)t}.
			\]
			Therefore choose $\tau' < 1/(A(0)(r+2d+1))$ and $A = A(\tau')$.
		\end{proof}
	\end{lemma} 
	\begin{lemma}
		There exists a $\tau>0$ and a constant $B>0$ such that
		\[
		|v_{ij}|^2 \le B . 
		\]
		at every $(x,t) \in \mathbb{R}^d\times [0,\tau]$.
		
		\begin{proof}
			
			Consider
			\[
			Q = v_{ij}^2 - B
			\]
			where $B(t)$ is a positive function of $t$ to be determined later. Suppose $Q = 0$ for the first time at $(x_0,t_0)$, $t_0 < \tau'$. At this point, 
			We have
			\begin{align*}
				0\le Q_t 
				&= vQ_{kk} + 2rv_kQ_k + 2v_{ij}^2v_{kk} - 2vv_{ijk}^2 \\
				&\quad
				+
				2v_{ij}v_{i}v_{kkj} + 2v_{ij}v_jv_{kki} + 4rv_{ij}v_{ki}v_{kj} - B'\\
				&\le vQ_{kk} + 2rv_kQ_k - 2vv_{ijk}^2 \\ &\quad +C|Dv||D^2v||D^3v| + C|D^2v|^3 - B'\\
				&\le 
				vQ_{kk} + 2rv_kQ_k - 2vv_{ijk}^2 \\ &\quad
				+ \frac{v}{2}|D^3v|^2 + C\frac{|Dv|^2}{v}|D^2v|^2 + C|D^2v|^3 -B'\\
				&\le  vQ_{kk} + 2rv_kQ_k + C\frac{|Dv|^2}{v}B + CB^{3/2}-B'			 
			\end{align*}
			where $C$ is a constant depending only on $m, d$ and $r = (m-1)^{-1}$. We want to choose $B$ such that
			\[
			B' > CB^{3/2} + C\frac{|Dv|^2}{v}B. 
			\]
			At the point $(x_0,t_0)$, since $t_0 \le \tau'$,
			\[
			0 \le C\frac{|Dv|^2}{v}B + CB^{3/2}-B' \le CAB + CB^{3/2} - B'.
			\]
			Choose $B$ that satisfies
			\[
			B' > CA(B + B^{3/2}).
			\]
			For example, 
			\[
			B(t) = \frac{1}{(a-CAt)^2}
			\]
			with $a$ chosen so that $1/a^2 > \sup_{x \in \mathbb{R}^d} |(v_{0,n}^{(\varepsilon)})_{ij}|^2$ for all $\varepsilon$ and $n$. 
			$\sup_{x \in \mathbb{R}^d} |(v_{0,n}^{(\varepsilon)})_{ij}|^2$ is bounded independent of $\varepsilon$ and $n$ because
			\[
			\begin{aligned}
				|D^2(\phi_n v_0)|^2 
				&= 
				\sum_{i,j}(({\phi_n})_{ij}v_0 + 2{\phi_n}_i{v_0}_j+{v_0}_{ij}\phi_n )^2 \\
				&\le \sum_{i,j}{v_0}_{ij}^2+ 2\sum_{i,j}{v_0}_{ij}(({\phi_n})_{ij}v_0 + 2{\phi_n}_i{v_0}_j) + \sum_{i,j} (({\phi_n})_{ij}v_0 + 2{\phi_n}_i{v_0}_j)^2\\
				&< \sum_{i,j}{v_0}_{ij}^2+C'|{v_0}_{ij}| + C'.
			\end{aligned}
			\]
			for some $C'$ independent of $\varepsilon$ and $n$ by Proposition 2. 
			Choose $\tau$ such that 
			\[
			\tau < \tau' \qquad \text{and} \qquad \tau < \frac{a}{CA}.
			\] 
			Now we have proved that
			\[
			|{v_{n}^{(\varepsilon)}}_{ij}|^2 < B
			\tag{3.1}
			\]
			for $t \in [0,\tau]$.
			with $B$ independent of $\varepsilon$ and $n$. Therefore by (3.1) and Lemma 3 we have
			  \[
			  -\sqrt{B'}
			  \le
			  \partial_t\log v_n^{(\varepsilon)}
			  =
			  \Delta v_n^{(\varepsilon)}
			  +r\frac{|Dv_n^{(\varepsilon)}|^2}{v_n^{(\varepsilon)}}
			  \le
			  \sqrt{B'}+rA.
			  \]
			Integrating, 
			  \[
			  e^{-\sqrt{B'}\,t}v_{0,n}^{(\varepsilon)}
			  \le v_n^{(\varepsilon)}(x,t)
			  \le
			  e^{(\sqrt{B'}+rA)t}v_{0,n}^{(\varepsilon)}
			  \]
			for every $(x,t) \in \mathbb{R}^d\times [0,\tau]$. For each fixed $(x,t)$, there exists $N$ such that $v_n^{(\varepsilon)}(x,t) \ge c > 0$ for every $n > N$ and $\varepsilon>0$. By standard local parabolic theory, we have $v_n^{(\varepsilon)} \rightarrow v$ smoothly as $\varepsilon \rightarrow 0$ and $n \rightarrow \infty$. Therefore, 
			\[
			|v_{ij}|^2 \le B.
			\]
			for $t \in [0,\tau]$.
		\end{proof}
	\end{lemma}
	Similarly we can also let $\varepsilon \rightarrow 0$ and $n \rightarrow \infty$ in Lemma 3: Since 
	$
	((v^{(\varepsilon)}_{n})_i)^2/v^{(\varepsilon)}_{n} < A
	$ for $A$ independent of $\varepsilon$ and $n$, we have $(v_i)^2/v \le A$ in $\mathbb{R}^d\times [0,\tau]$. 
	\vspace{\baselineskip}
	
	Now we want to show that the root convexity $D^2(\sqrt{v}) \ge 0$ is preserved for $t \in [0, \tau]$. We want to show that the smallest eigenvalue $\lambda_1$ of the symmetric matrix 
	\[
	A_{ij} = v_{ij}-\frac{v_iv_j}{2v}+\psi\sqrt{v}I_{ij},
	\]
	where $\psi(t)>0$ is to be determined later, 
	is positive at every $(x,t) \in \mathbb{R}^d \times [0,\tau]$. As $|x| \rightarrow \infty$, $\lambda_1  > 0$. Assume that there is an interior point $(x_0,t_0)$ where $\lambda_1 = 0$ for the first time. Let $\mu$ denote the multiplicity of $\lambda_1$ at $(x_0,t_0)$. Choose coordinates so that $A_{ij}$ is diagonal with eigenvalues $\lambda_i = A_{ii}$ at $(x_0,t_0)$. We want to compute the evolution of $\lambda_1(x,t)$, but this function is not smooth in general. Therefore we use a viscosity type argument:
	\begin{claim}
		Let $\varphi$ be a smooth function such that $0\le\varphi < \lambda_1$ in $\mathbb{R}^d\times [0,t_0)$ and $\varphi = \lambda_1$ at $p = (x_0,t_0)$. Then we have $(\varphi)_t \ge (A_{11})_t$.
		\begin{proof}
			Define $\gamma(s) = (x_0, t_0-s)$. and let $v(s)$ be a vector field along $\gamma$ such that $v(0) \in \text{span} \{e_1,...,e_{\mu} \}$ and $v'(0) \in \text{span} \{e_{\mu+1},...,e_d \}$. Then the function $s \mapsto \langle Av(s),v(s)\rangle - \varphi(\gamma(s))|v(s)|^2 $ is 0 as $s = 0$ and  $>0$ as $s>0$. Therefore
			\begin{align*}
				0 &\le \left. \frac{d}{ds}\left( \langle A(\gamma(s))v(s),v(s)\rangle - \varphi(\gamma(s))|v(s)|^2 \right)
				\right|_{s = 0}\\
				&= \langle (\left. DA_{ij}\right|_p\cdot\gamma'(0))v(0),v(0)\rangle + 2\langle A(\gamma(0))v'(0), v(0)\rangle - \left. \frac{d}{ds}\right|_{s=0}\varphi(\gamma(s))|v(0)|^2\\
				&\quad - \varphi(\gamma(s))\langle v'(0),v(0) \rangle\\
				&= -\langle \left.(A_{ij})_t\right|_pv(0),v(0)\rangle + \left. \varphi_t\right|_p|v(0)|^2
			\end{align*}
			Hence $\langle \left.(A_{ij})_t\right|_pv(0),v(0)\rangle \le \left. \varphi_t\right|_p|v(0)|^2$ for any $v(0) \in \text{span} \{e_1,...,e_{\mu}\}$. Let $v(0) = e_1$, we get $(\varphi)_t \ge (A_{11})_t$.
		\end{proof}
	\end{claim}
	Choose $\varphi = 0$, by Lemma 5 in \cite{BrendleChoiDaskalopoulos2017} and Claim 1, at $(x_0,t_0)$ we have 
	\[
	0 = (\varphi)_k = (A_{11})_k,\quad 0 = (\varphi)_t \ge (A_{11})_t,
	\]
	\[
	(A_{1m})_k = 0 \quad \text{for} \quad 1<m\le\mu,
	\]
	and
	\[
	0 = (\varphi)_{kk} \le (A_{11})_{kk} + 2\sum_{m>\mu}\frac{A^2_{m1k}}{\lambda_1-\lambda_m}.
	\]
	Hence
	\begin{align*}
		0  = (\varphi)_t - v(\varphi)_{kk}
		&\ge \left(v_{11}-\frac{v_1^2}{2v}+\psi \sqrt{v}\right)_t - v\left(v_{11}-\frac{v_1^2}{2v}+\psi \sqrt{v}\right)_{kk}\\
		&\quad-2\sum_{m>\mu}\frac{v}{\lambda_1-\lambda_m}\left[\left(v_{m1}-\frac{v_mv_1}{2v}\right)_k\right]^2.
		\tag{1}
	\end{align*}
	Compute
	\[
	v_{1t} = vv_{kk1}+v_1v_{kk}+2rv_kv_{k1}
	\]
	and
	\[
	v_{11t} = vv_{kk11} + 2v_1v_{kk1} + v_{11}v_{kk} + 2rv_kv_{k11} + 2rv_{k1}^2.
	\]
	Hence
	\[
	(v_{11})_t - v(v_{11})_{kk} = 2v_1v_{kk1} + v_{11}v_{kk} + 2rv_kv_{k11} + 2rv_{k1}^2.
	\tag{2}
	\]
	On the other hand,
	\[
	\left(\frac{v_1^2}{2v}\right)_t = v_1v_{kk1} + \frac{v_1^2v_{kk}}{2v} + \frac{2rv_1v_kv_{k1}}{v} - \frac{rv_1^2v_k^2}{2v^2}
	\]
	and
	\[
	\left(\frac{v_1^2}{2v}\right)_k = \frac{v_1v_{1k}}{v} - \frac{v_1^2v_k}{2v^2},
	\]
	
	\[
	\left(\frac{v_1^2}{2v}\right)_{kk} = \frac{v_{1k}^2}{v} + \frac{v_1v_{1kk}}{v} - \frac{2v_1v_kv_{1k}}{v^2} - \frac{v_1^2v_{kk}}{2v^2} + \frac{v_1^2v_k^2}{v^3}.
	\]
	Hence
	\[
	\left(\frac{v_1^2}{2v}\right)_t - v\left(\frac{v_1^2}{2v}\right)_{kk} = \frac{v_1^2v_{kk}}{v} + \frac{2(r+1)v_1v_kv_{1k}}{v} - \frac{(r+2)v_1^2v_k^2}{2v^2} - v_{1k}^2.
	\tag{3}
	\]
	Next, 
	\[
	\left(
	\sqrt{v}\,\psi
	\right)_t
	=
	\frac{vv_{kk}+rv_k^2}{2\sqrt{v}}\psi
	+
	\sqrt{v}\,\psi_t ,
	\]
	and
	\[
	v(\sqrt{v}\,\psi)_{kk}
	=
	\left(
	\frac{v_k\psi}{2\sqrt{v}}
	\right)_k
	=
	\left(\frac{\sqrt{v} v_{kk}}{2}-\frac{v_k^2}{4\sqrt{v}}\right)\psi.
	\]
	Hence
	\[
	(
	\sqrt{v}\,\psi
	)_t
	-
	v(\sqrt{v}\,\psi)_{kk}
	=
	\left(\frac{r}{2}+\frac{1}{4}\right)\frac{v_k^2}{\sqrt{v}}\psi + \sqrt{v}\psi_t
	\tag{4}
	\]
	Combining $(2)$, $(3)$ and $(4)$ we get
	\begin{align*}
		&\quad \left(v_{11}-\frac{v_1^2}{2v}+\psi \sqrt{v}\right)_t - v\left(v_{11}-\frac{v_1^2}{2v}+\psi \sqrt{v}\right)_{kk}\\
		&= 2v_1v_{kk1} + v_{11}v_{kk} + 2rv_kv_{k11} + (2r+1)v_{k1}^2 - \frac{v_1^2v_{kk}}{v} - \frac{2(r+1)v_1v_kv_{1k}}{v} + \frac{(2+r)v_1^2v_k^2}{2v^2}.\\
		&\quad + \left(\frac{r}{2}+\frac{1}{4}\right)\frac{v_k^2}{\sqrt{v}}\psi + \sqrt{v}\psi_t
	\end{align*}
	At $(x_0,t_0)$ we have $(\lambda_1)_k = 0$ so
	\[
	v_{11k} = \frac{v_1v_{1k}}{v} - \frac{v_1^2v_k}{2v^2} - \frac{\psi v_k}{2\sqrt{v}}
	\]
	Rewriting the terms involving $v_{11k}$ we get
	\begin{align*}
		&\quad \left(v_{11}-\frac{v_1^2}{2v}+\psi \sqrt{v}\right)_t - v\left(v_{11}-\frac{v_1^2}{2v}+\psi \sqrt{v}\right)_{kk}\\
		&= 2v_1\left[v_{kk1}-\frac{v_{kk}v_1}{2v}-\frac{v_kv_{k1}}{2v}+\frac{ v_k^2v_1}{2v^2}\right] + \frac{v_{kk}v_1^2}{v} + \frac{v_1v_kv_{k1}}{v} - \frac{v_k^2v_1^2}{v^2}\\
		&\quad + v_{11}v_{kk} + (2r+1)v_{k1}^2 - \frac{v_1^2v_{kk}}{v} - \frac{2v_1v_kv_{1k}}{v} + \frac{(2-r)v_1^2v_k^2}{2v^2}\\
		&\quad +   \left(\frac{r}{2}+\frac{1}{4}\right)\frac{v_k^2}{\sqrt{v}}\psi + \sqrt{v}\psi_t - r\frac{\psi v_k^2}{\sqrt{v}}\\
		&=
		2v_1\sum_{k>\mu}\left[v_{k1}-\frac{v_kv_1}{2v}\right]_k + v_{11}v_{kk} +(2r+1)v_{k1}^2 - \frac{v_1v_kv_{1k}}{v} - \frac{rv_1^2v_k^2}{2v^2} - r\frac{\psi v_k^2}{\sqrt{v}}\tag{5} \\
		&\quad +   \left(\frac{r}{2}+\frac{1}{4}\right)\frac{v_k^2}{\sqrt{v}}\psi -\frac{\psi v_1^2}{\sqrt{v}}+ \sqrt{v}\psi_t 
	\end{align*}
	In the second equality we used that $(A_{1k})_k = (v_{1k}-v_1v_k/(2v))_k = 0$ for $1<k\le \mu$. 
	At $(x_0,t_0)$, we have for $k = 1, \dots, d$, 
	\[
	v_{kk} = \frac{v_k^2}{2v} - \psi \sqrt{v}+ \lambda_k
	\]
	and for $k \neq 1$, 
	\[
	v_{1k} = \frac{v_kv_1}{2v}.
	\]
	Since $\lambda_1 = 0$, at $(x_0,t_0)$ we have 
	\[
	\psi \sqrt{v} = \frac{v_1^2}{2v}-v_{11}
	\]
	and
	\[
	\lambda_k = v_{kk}-\frac{v_k^2}{2v}-v_{11}+\frac{v_1^2}{2v}.
	\]
	Since $|D^2v|^2$ and $|Dv|^2/v$ are uniformly bounded for $(x,t) \in \mathbb{R}^d \times [0,\tau]$, $\sum_{k \ge 1} \lambda_k$ at $(x_0,t_0)$ is bounded by some constant not depending on $(x_0,t_0)$. \\
	Hence
	\begin{align*}
		v_{11}v_{kk} &= 	\left(\frac{v_1^2}{2v} - \psi \sqrt{v} +\lambda_1\right)\sum_{k = 1}^d\left(\frac{v_k^2}{2v} - \psi \sqrt{v} +\lambda_k\right)\\
		&\ge\left(\frac{v_1^2}{2v} +\lambda_1\right)\sum_{k = 1}^d\left(\frac{v_k^2}{2v}  +\lambda_k\right)
		-
		\psi \sqrt{v}\left(\frac{dv_1^2}{2v}+\sum_{k = 1}^d\left(\frac{v_k^2}{2v}+\lambda_k\right)\right)\\
		&=\sum_{k=1}^d \frac{v_1^2v_k^2}{4v^2} + \frac{v_1^2}{2v}\sum_{k=1}^d \lambda_k - \psi \sqrt{v}\left(\frac{dv_1^2}{2v}+\sum_{k = 1}^d\left(\frac{v_k^2}{2v}+\lambda_k\right)\right).
	\end{align*}
	For $k=1$, 
	\begin{align*}
		(2r+1)v_{11}^2 =  (2r+1)  \left(\frac{v_1^2}{2v} - \psi \sqrt{v} +\lambda_1\right)^2 &\ge (2r+1)\left(\frac{v_1^2}{2v}+\lambda_1\right)^2\\
		&\quad - 2(2r+1)\psi \sqrt{v} \left(\frac{v_1^2}{2v}+\lambda_1\right)
	\end{align*}
	For $k>1$, 
	\begin{align*}
		(2r+1)v_{1k}^2 = (2r+1)\frac{v_1^2v_k^2}{4v^2}
	\end{align*}
	For $k=1$, 
	\begin{align*}
		-\frac{v_1v_kv_{1k}}{v} = -\frac{v_1^2v_{11}}{v} = -\frac{v_1^2}{v} \left(\frac{v_1^2}{2v} - \psi \sqrt{v} +\lambda_1\right) &\ge -\frac{v_1^4}{2v^2}
		+ \frac{v_1^2}{2v}\psi \sqrt{v}\\
		&\ge -\frac{v_1^4}{2v^2}
	\end{align*}
	For $k>1$,
	\begin{align*}
		-\frac{v_1v_kv_{1k}}{v} = -\frac{v_1^2v_k^2}{2v^2}.
	\end{align*}
	Therefore
	\begin{align*}
		(5) &\ge  2v_1\sum_{k>\mu}\left[v_{k1}-\frac{v_kv_1}{2v}\right]_k +\sum_{k=1}^d \frac{v_1^2v_k^2}{4v^2} + \frac{v_1^2}{2v}\sum_{k=1}^d \lambda_k - \psi \sqrt{v}\left(\frac{dv_1^2}{2v}+\sum_{k = 1}^d\left(\frac{v_k^2}{2v}+\lambda_k\right)\right)\\
		&\quad + (2r+1)\left(\frac{v_1^2}{2v}+\lambda_1\right)^2 - 2(2r+1)\psi \sqrt{v}\left(\frac{v_1^2}{2v}+\lambda_1\right)\\
		&\quad + (2r+1)\sum_{k>1}\frac{v_1^2v_k^2}{4v^2} -\frac{v_1^4}{2v^2} - \sum_{k>1}\frac{v_1^2v_k^2}{2v^2} - r\frac{\psi v_k^2}{\sqrt{v}}\\
		&\quad - \sum_{k=1}^d
		\frac{rv_1^2v_k^2}{2v^2} -\frac{\psi v_1^2}{\sqrt{v}} +
		\left(\frac{r}{2}+\frac{1}{4}\right)\frac{v_k^2}{\sqrt{v}}\psi + \sqrt{v}\psi_t
	\end{align*}
	\begin{align*}
		&=
		2v_1\sum_{k>\mu}\left[v_{k1}-\frac{v_kv_1}{2v}\right]_k + \frac{v_1^2}{2v}\sum_{k=1}^d \lambda_k - 
		\psi \sqrt{v}\left(\frac{dv_1^2}{2v}+\sum_{k = 1}^d\left(\frac{v_k^2}{2v}+\lambda_k\right)\right) - r\frac{\psi v_k^2}{\sqrt{v}}\\
		&\quad- 
		2(2r+1)\psi \sqrt{v} \left(\frac{v_1^2}{2v}+\lambda_1\right)
		+
		\left( \frac{r}{2}+\frac{1}{4}\right)\frac{v_k^2}{\sqrt{v}}\psi -\frac{\psi v_1^2}{\sqrt{v}} + \sqrt{v}\psi_t \\
		&\ge
		\sum_{k>\mu}\frac{2v}{\lambda_1-\lambda_k}\left(
		\left[v_{k1}-\frac{v_kv_1}{2v}\right]_k
		\right)^2 + \frac{v_1^2}{2v}\sum_{k>1}(\lambda_1-\lambda_k)
		+\frac{v_1^2}{2v}\sum_{k=1}^d \lambda_k- r\frac{\psi v_k^2}{\sqrt{v}}\\
		&\quad - 
		\psi \sqrt{v}\left(\frac{nv_1^2}{2v}+\sum_{k = 1}^d\left(\frac{v_k^2}{2v}+\lambda_k\right)\right)\\
		&\quad- 
		2(2r+1)\psi \sqrt{v} \left(\frac{v_1^2}{2v}+\lambda_1\right) -\frac{\psi v_1^2}{\sqrt{v}}
		+ 
		\left( \frac{r}{2}+\frac{1}{4}\right)\frac{v_k^2}{\sqrt{v}}\psi + \sqrt{v}\psi_t
	\end{align*}
	
	Then there exists a $C$ depending only on $m, d$ and $v_0$ such that all terms involving $\psi$ are bounded below by $-C\psi \sqrt{v}$. By choosing $\psi$ such that $\psi_t>C\psi$, for example $\psi = \varepsilon e^{(C+1)t}$, and combining with $(1)$ we get a contradiction. Therefore
	\[
	A_{ij} = v_{ij}-\frac{v_iv_j}{2v}+\psi\sqrt{v}I_{ij} > 0
	\]
	or
	\[
	\frac{1}{\sqrt{v}}\left(v_{ij}-\frac{v_iv_j}{2v}\right)+\psi I_{ij}> 0
	\]
	Finally let $\varepsilon \rightarrow 0$, 
	\[
	\frac{1}{\sqrt{v}}\left(v_{ij}-\frac{v_iv_j}{2v}\right) > 0.
	\]

	\subsection{Proof of Main Theorem}
	
	We have proved that if initially 
	\[
	\sup_{x\in \mathbb{R}^d} \Delta v< \infty,
	\]
	then the root convexity is preserved for some time $\tau>0$. Now we prove Main Theorem.\\
	\begin{lemma}
		Let $T>0$. Suppose that $\Delta v$ is bounded for $t \in [0,T)$ and $v_0$ is root convex. Then $v$ is root convex for all $t \in [0,T]$.
		\begin{proof}
			Since $v_0$ is root convex and $\Delta v_0$ is bounded, $|{v_0}_{ij}|^2$ is bounded. By section 3.2 we can prove that $v$ is root convex for short time. Define
			\[
			\tau := \sup\left\{ \delta: v \, \text{is root convex for } t \in [0,\delta]  \right\}.
			\]
			If $\tau = T$ then we are done. If $\tau < T$, at $t = \tau$, $v$ is root convex and $\Delta v$ is bounded so again by section 3.2 we can prove that $v$ is root convex for $t \in [0,\tau+\tau']$. Contradiction. 
		\end{proof}
	\end{lemma}
	
	\begin{lemma}
		Suppose that $v_0$ is root convex, and 
		\[
		\sup_{x\in \mathbb{R}^d} \Delta v_0 = K < \infty. 
		\]
		Then
		\[
		\sup_{x\in \mathbb{R}^d} \Delta v (t) \le \frac{1}{1/K-(2r+1)t}
		\] 
		for
		\[
		t \in [0, \min \left\{ T_r, \frac{1}{(2r+1)K} \right\})
		\].
		
		\begin{proof}	
			Choose $K<K'<K+1$. Consider the quantity
			\[
			Q = v_{ii} - B(t) - \varepsilon t (|x|^2+1)
			\]
			We have proved that $|D^2v|^2$ is bounded for some short time $T$ and therefore $v$ is root convex for $t \in [0,T]$, so $\Delta v$ is bounded for $t \in [0,T]$. Then as $|x| \rightarrow \infty$, $Q<0$ for every $t >0$. $B(t)$ is a positive function of $t$ to be determined. Suppose $Q = 0$ for the first time at $(x_0,t_0)$. At this point, we have
			\begin{align*}
				0\le Q_t 
				&\le vQ_{kk} + (2r+2)v_kQ_k + v_{ii}^2 + 2rv_{ik}^2 + 4(r+1)\varepsilon t v_ix_i + 2d\varepsilon t v \\
				&\quad
				-B'-\varepsilon (|x|^2+1)\\
				&\le vQ_{kk} + (2r+2)v_kQ_k + (2r+1)v_{ii}^2 + 4(r+1)\varepsilon t v_ix_i + 2d\varepsilon t v \\
				&\quad
				-B'-\varepsilon (|x|^2+1) \quad (\text{by convexity}\,  \sum_{i,k=1}^d v_{ik}^2 \le \left(\sum_{i=1}^d v_{ii}\right)^2)\\
				&\le  (2r+1)(B+\varepsilon t (|x|^2+1))^2  \\ 
				&\quad + 4(r+1)\varepsilon t v_ix_i + 2d\varepsilon t v -B'-\varepsilon(|x|^2+1)\\
				&\le  (2r+1)B^2 - B'\\		 
			\end{align*}
			The last inequality is because we can choose $\delta < T$ small enough such that 
			\[
			t(\frac{4(r+1)v_ix_i}{|x|^2+1} + \frac{2dv}{|x|^2+1} + 2B(t)+t\varepsilon (|x|^2+1))<1
			\]
			for $t < \delta$,
			since at $(x_0,t_0)$, $t\varepsilon (|x|^2+1) < \Delta v$, and $\Delta v$ is bounded in $\mathbb{R}^d \times [0,\delta]$.  Define
			\[ 
			B(t) = \frac{1}{1/K' - (2r+1)t}.
			\]
			$\delta$ can be chosen independent of $\varepsilon$ and $K'-K$.
			Let $\varepsilon \rightarrow 0$. 
			Therefore
			\[
			\sup_{x\in \mathbb{R}^d} \Delta v(t) \le  \frac{1}{1/K' - (2r+1)t}.
			\]
			for every $t \in [0,\delta]$.  Since $K<K'<K+1$ is arbitrary, we have
			\[
			\sup_{x\in \mathbb{R}^d} \Delta v (t)\le  \frac{1}{1/K - (2r+1)t}.
			\]
			for every $t \in [0,\delta]$.\\
			Define
			\[
			\tau_{\max} := \sup\Bigl\{\delta: \sup_{x\in \mathbb{R}^d} \Delta v (t)\le  \frac{1}{1/K - (2r+1)t} \quad \text{for} \quad t \in [0,\delta] \Bigr\}
			\]
			By argument above $\tau_{\max}>0$. Suppose, by contradiction, that
			$\tau_{\max} < 1/(K(2r+1))$. 
			Then
			\[
			\sup_{x\in \mathbb{R}^d} \Delta v (t)\le  \frac{1}{1/K - (2r+1)\tau_{\max}} 
			\]
			for $t<\tau_{\max}$.
			Since $v$ is smooth we have
			\[
			\sup_{x\in \mathbb{R}^d} \Delta v(\tau_{\max}) \le  \frac{1}{1/K - (2r+1)\tau_{\max}}
			\] 
			and $v$ is root convex at $t = \tau_{\max}$ because $\sup_{x\in \mathbb{R}^d} \Delta v$ is bounded in $[0,\tau_{\max}]$ and root convexity is preserved as long as $\sup_{x\in \mathbb{R}^d} \Delta v$ is bounded.  Then 
			consider the quantity
			\[
			Q = v_{ii} - B(t) - \varepsilon (t-\tau_{\max}) (|x|^2+1).
			\]
			Therefore, there exists a $\tau'>0$ such that
			\[
			\sup_{x\in \mathbb{R}^d} \Delta v(t) \le  \frac{1}{1/K - (2r+1)t}.
			\]
			for $t \in [\tau_{\max}, \tau_{\max}+\tau']$. Contradiction. 
		\end{proof}
	\end{lemma}
	Then Main Theorem follows from Lemma 5 and Lemma 6, together with the existence result Theorem 2.

\end{document}